\documentclass[12pt]{amsart}

\usepackage{amssymb}
\usepackage{enumerate}
\usepackage[dvips]{graphicx}
\DeclareGraphicsExtensions{.eps,.art,.ART,.ps}

\usepackage{amsmath, amsthm, amsfonts, color}

\newcommand{\R}{\mathbb{R}}      
\newcommand{\ds}{\displaystyle}

\newcommand{\dint}{\ds\int}

\newcommand{\dsum}{\ds\sum}

\newcommand{\eqskip}{ \vspace*{2mm}\\ }

\newtheorem{theorem}{Theorem}
\newtheorem{thm}{Theorem}[section]

\newtheorem{lemma}[thm]{Lemma}

\theoremstyle{definition}

\theoremstyle{remark}

\newtheorem{rmk}[thm]{Remark}

\def\tht{\theta}
\def\Om{\Omega}
\def\om{\omega}
\def\e{\varepsilon}
\def\g{\gamma}

\def\l{\lambda}
\def\p{\partial}
\def\D{\Delta}

\def\E{\mbox{\rm e}}
\def\a{\alpha}
\def\b{\beta}

\def\d{\delta}
\def\L{\Lambda}

\def\vs{\varsigma}

\def\vr{\varrho}

\def\Odr{\mathcal{O}}
\def\H{W_2}

\def\di{\,\mathrm{d}}

\def\xm{\overline{x}}

\DeclareMathOperator{\Dom}{\mathcal{D}}

\numberwithin{equation}{section}

\title[Asymptotics for Dirichlet eigenvalues in $\mathbb{R}^d$]
{Asymptotics of Dirichlet eigenvalues and eigenfunctions of the
Laplacian on thin domains in $\mathbb{R}^d$}
\author{Denis Borisov \and Pedro Freitas}
\address{
Department of Physics and Mathematics, Bashkir State Pedagogical
University, October rev. st., 3a, 450000, Ufa, Russia
}\email{borisovdi@yandex.ru}
\address{Department of Mathematics,
Faculdade de Motricidade Humana (TU Lisbon) {\rm and} Group of
Mathematical Physics of the University of Lisbon\\ Complexo
Interdisciplinar, Av.~Prof.~Gama Pinto~2\\ P-1649-003 Lisboa,
Portugal}\email{freitas@cii.fc.ul.pt}
\date{\today}

\thanks{D.B. was partially supported by RFBR
and gratefully acknowledges the support from Deligne
2004 Balzan prize in mathematics and the grant of the President of
Russia for young scientists 
and for Leading Scientific Schools (NSh-2215.2008.1). P.F. was
partially supported by POCTI/POCI2010, Portugal, and by project
LC06002 of the Czech Ministry of Education, Youth and Sports. Part
of this work was done while P.F. was visiting the Doppler Institute
in Prague, and he would like to thank the people there, and in
particular P. Exner and D. Krej\v{c}i\v{r}{\'\i}k, for their
hospitality.}

\begin{document}

\allowdisplaybreaks \maketitle
\begin{abstract}
We consider the Laplace operator with Dirichlet boundary
conditions on a domain in $\mathbb{R}^d$ and study the effect that
performing a scaling in one direction has on the eigenvalues and
corresponding eigenfunctions as a function of the scaling
parameter around zero. This generalizes our previous
results in two dimensions and, as in that case, allows us to
obtain an approximation for Dirichlet eigenvalues
for a large class of domains, under very mild assumptions. As an
application, we derive a three--term asymptotic expansion for the first
eigenvalue of $d-$dimensional ellipsoids.
\end{abstract}

\section{Introduction}

In his $1967$ paper~\cite{jose} Joseph studied families
of domains indexed by one parameter to obtain perturbation formulae approximating
eigenvalues in a neighbourhood of a given domain. Within this context, he derived an
elegant expression for the first eigenvalue of ellipses parametrized by their
eccentricity $e$, namely,
\begin{equation}\label{josephellipses}
\begin{array}{lll}
\lambda_{1}(e) & = & \lambda_{1}-\frac{\ds \lambda_{1}}{\ds 2}e^{2}-\frac{\ds \lambda_{1}}{\ds 16}
\left(3-\frac{\ds \lambda_{1}}{\ds 2}\right)e^4\eqskip
& & \hspace*{1cm}-\frac{\ds \lambda_{1}}{\ds 32}
\left(3-\frac{\ds \lambda_{1}}{\ds 2}\right)e^6+\Odr(e^8), \;\; \mbox{
as } e\to0,
\end{array}
\end{equation}
where $\lambda_{1}=\lambda_{1}(0)$ is the first eigenvalue of the disk -- to obtain the
eigenvalue of ellipses of, say, area $\pi$, for instance, this should be divided by $\sqrt{1-e^2}$
and $\lambda_{1}(0)$ be the corresponding value for the disk. The coefficient of order $e^6$
in Joseph's paper is actually incorrect -- we are indebted to M. Ashbaugh
for pointing this out to us, and also for mentioning Henry's book~\cite{henr} where this has been corrected.
Although in principle quite general, the approach used by Joseph yields formulae
which, in the case of domain perturbations, will allow us to obtain explicit
asymptotic expansions only in very special cases such as that of ellipses above. The failure to
obtain these expressions may be the case even when the eigenvalues and eigenfunctions of the
original domain are known, as this does not necessarily mean that the coefficients appearing in the expansion may be computed in closed form. An example of this is the perturbation of a rectangle into
a parallelogram, which Joseph considered as an example of what he called ``pure shear.''

With the purpose of obtaining approximations that can be computed explicitly, in a previous
paper we considered instead the scaling of a given two dimensional domain in one direction
and studied the resulting singular perturbation as the domain approached a segment in the limit~\cite{BF}.
This approach may, of course, have the disadvantage that we might now be starting too
far from the original domain. However, it allows for the explicit derivation of the coefficients
in the expansion in terms of the functions defining the boundary of the domain. As was to be expected,
and can be seen from the examples given in that paper, these four--term approximations are quite accurate
close to the thin limit. A more interesting feature of this approach is that in some cases it also allows us to approximate eigenvalues quite well away from this limit, as may be seen from the following examples. The application of our formula to the ellipses considered above yields
\begin{equation}\label{diskasympt}
\lambda_{1}(\e) = \frac{\ds \pi^{2}}{\ds
4\e^2}+\frac{\ds \pi}{\ds 2\e} +\frac{\ds 3}{\ds 4} + \left(\frac{\ds 11}{\ds
8\pi}+\frac{\ds \pi}{\ds 12}\right)\e + \Odr(\e^2), \;\; \mbox{
as } \e\to+0,
\end{equation}
where we now considered ellipses of radii $1$ and $\e$, $\e$ being the stretch factor. The error in
the approximation is comparable to that in Joseph's formula, except that equation~(\ref{diskasympt})
is more accurate closer to the thin limit while~(\ref{josephellipses}) provides better approximations
near the circle. This is also an advantage, since it is natural for numerical methods to perform better away from the thin limit, but to have more difficulties the closer they are to the singular
case, suggesting that our formulae may also be useful for checking numerical methods close to the
limit case.

As another application we mention the case of the lemniscate
\[
\left( x_{1}^2+ x_{2}^2\right)^2 = x_{1}^2-x_{2}^2.
\]
for which we have
\[
\lambda_{1}(\e) = \frac{\ds 2\pi^{2}}{\ds
\e^2}+\frac{\ds 2\sqrt{3}\pi}{\ds \e} +\frac{\ds 97}{\ds 24} + \left(\frac{\ds 593
}{\ds 64\sqrt{3}\pi}+\frac{\ds \sqrt{3}\pi}{\ds 4}\right)\e + \Odr(\e^2), \;\; \mbox{
as } \e\to+0,
\]
yielding an error at $\e$ equal to one which is in fact smaller than in the case of the
disk above. For details, see~\cite{BF}.

In the present paper we extend the results in~\cite{BF} to general
dimension, in the sense that we now consider domains in
$\mathbb{R}^d$ which are being scaled in one direction and approach
a $(d-1)$--dimensional set in the limit as the stretch parameter
goes to zero. Due to the increase in complexity in the corresponding
formulae as a consequence of the fact that we are now considering
arbitrary dimensions, we only obtained the first three non-zero
coefficients in the asymptotic expansion of the principal
eigenvalue. However, because of smoothness assumption near the point
of global maximum, these include the coefficients of the two
unbounded terms plus the constant term in the expansion -- we know
from the two--dimensional case that lack of smoothness at the point
of maximum will yield other intermediate powers of
$\e$~\cite{frei,frso}. 

As an example, we obtain an expansion for the first eigenvalue of the general $d-$dimensional
ellipsoid
\[
\mathcal{E}=\left\{ (x_{1},\dots,x_{d})\in\R^{d}:\left(\frac{\ds x_{1}}{\ds a_{1}}\right)^2
+\dots+\left(\frac{\ds x_{d}}{\ds a_{d}}\right)^2\leq 1\right\},
\]
where the $a_{i}'$s are positive real numbers. If we choose as projecting hyperplane that which is orthogonal to the $x_{d}$ axis we obtain
\begin{equation}\label{lellipsoid}
\begin{array}{lll}
\lambda_{1}\left(\mathcal{E}_{\e}\right) & = & \frac{\ds \pi^2}{\ds 4a_{d}^2 \e^2}+
\frac{\ds \pi}{\ds 2a_{d}\e}\dsum_{i=1}^{d-1}\frac{\ds 1}{\ds a_{i}}\eqskip
& & \hspace*{5mm}+\frac{\ds 1}{\ds 4}\left(3\dsum_{i=1}^{d-1}\frac{\ds 1}{\ds a_{i}^2}+
\frac{\ds 1}{\ds 2}\dsum_{i=1}^{d-1}\dsum_{j=i+1}^{d-1}\frac{\ds 1}{\ds a_{i}a_{j}}\right)\eqskip
& & \hspace*{15mm} +\Odr(\e^{1/2}),\quad \mbox{ as }\e\to+0.
\end{array}
\end{equation}

Besides the added complexity of the formulae, there are now extra
technical difficulties related to the fact that there may exist
multiple eigenvalues requiring a more careful approach. As in the
two dimensional case, the asymptotic expansions obtained depend on
what happens locally at the point of global maximum width. Also as
in that case, we cannot exclude the existence of a tail term
approaching zero faster than any power of $\e$. 
However, we conjecture that if the boundary of the domain is
analytic, then the expansions will actually correspond to the series
developments of the corresponding eigenvalues.

In the next section we establish the notation and state the main
results of the paper, which are then proved in
Sections~\ref{proofthm1} and~\ref{proofthm2}. In the last section,
and as an application, we derive the above expression for the first
eigenvalue of the ellipsoid.

\section{Statement of results\label{stat}}

Let $x=(x',x_d)$, $x'=(x_1,\ldots,x_{d-1})$ be Cartesian
coordinates in $\mathbb{R}^d$ and $\mathbb{R}^{d-1}$, respectively,
$d\geqslant 2$, and $\om\subset \mathbb{R}^{d-1}$ be a bounded
domain having $C^1$-boundary. By $h_\pm=h_\pm(x')\in
C(\overline{\om})$ we denote two arbitrary functions such that
$H(x'):=h_+(x')+h_-(x')\geqslant 0$ for $x'\in\om$. We consider the
thin domain defined by
\begin{equation*}
\Om_\e:=\{x: -\e h_-(x')<x_d<\e h_+(x'), x'\in\om\},
\end{equation*}
where $\e$ is a small positive parameter. We assume that the
function $H(x')$ attains its global maximum at a single point
$\xm\in\om$ and that there
exists a ball $B'_\d(\xm):=\{x': |x'-\xm|<\d\}$ such that
$h_\pm\in C^\infty(B'_\d(\xm))$. Let $H_0:=H(\xm)$ and
the Taylor expansions for $H$ and $h_-$ at $\xm$ read as follows
\begin{equation}\label{2.1}
H(x')=H_0+\sum\limits_{i=2k}^{\infty} H_i(x'-\xm), \quad
h_-(x')=h_0+\sum\limits_{i=1}^{\infty} h_i(x'-\xm),
\end{equation}
where $H_i$ and $h_i$ are homogeneous polynomials of order $i$,
$H_{2k}(x'-\xm)<0$ for $x'\not=\xm$, and $k\geqslant 1$.

Our purpose is to study the asymptotic behaviour of the
eigenvalues and eigenfunctions of the Dirichlet Laplacian
$-\D^{D}_{\Om_\e}$ in $\Om_\e$. Let $\chi=\chi(x')\in
C^\infty(\mathbb{R}^{d-1})$ be a non-negative cut-off function
equalling one as $|x'-\xm|<\d/3$ and vanishing for
$|x'-\xm|>\d/2$. Denote $\Om_\e^\d:=\Om_\e\cap \{x:
|x'-\xm|<\d\}$.

Let
\begin{equation*}
G_n:=-\D_{\xi'}-\frac{2\pi^2n^2 H_{2k}(\xi')}{H_0^3}
\end{equation*}
be an operator in $L_2(\mathbb{R}^{d-1})$. The spectrum of this
operator consists of countably many isolated eigenvalues of
finite multiplicity having only one accumulation point at
infinity \cite[Ch. I\!V, Sec. 46, Th. 1]{G}. By
\begin{equation*}
\L_{n,1}<\L_{n,2}\leqslant \L_{n,3}\ldots
\end{equation*}
we denote the eigenvalues of this operator arranged in
non-decreasing order and taking the multiplicities into account.
Denote by $\Psi_{n,m}$ the associated eigenfunctions
orthonormalized in $L_2(\mathbb{R}^{d-1})$. It follows from
\cite[Ch. V, Sec. 43, Th. 2]{G} that the functions $\Psi_{n,m}$
decay exponentially at infinity.

Our main results are the following. First, we obtain a two--parameter
description for the eigenvalues.

\begin{theorem}\label{th2.1}
Let $\L=\L_{n,M}=\L_{n,M+1}=\ldots=\L_{n,M+N-1}$ be a
$N$-multiple eigenvalue of $G_n$ for a given $n\in\mathbb{N}$.
Then there exist eigenvalues $\l_{n,m}(\e)$ of
$-\D^{D}_{\Om_\e}$, $m=M,\ldots,M+N-1$ taken counting
multiplicities whose asymptotics as $\e\to+0$ read as follows
\begin{gather}
\l_{n,m}(\e)=\e^{-2}c_0^{(n,m)}+\e^{-2}\sum\limits_{j=2k}^{\infty}
c_j^{(n,m)}\eta^j, \quad \eta:=\e^\a,\quad
\a:=\frac{1}{k+1},\label{2.2}
\\
c_0^{(n,m)}=\frac{\pi^2 n^2}{H_0^2},\quad
c_{2k}^{(n,m)}=\L,\label{2.3}
\end{gather}
and $-c_{2k+1}^{(n,m)}$ are the eigenvalues of the matrix with
the entries
\begin{equation*}
2\pi^2n^2 H_0^{-3}
(H_{2k+1}\Psi_{n,m},\Psi_{n,l})_{L_2(\mathbb{R}^{d-1})},\quad
m,l=M,\ldots,M+N-1.
\end{equation*}
The remaining coefficients are determined by
Lemmas~\ref{lm3.7} and~\ref{lm3.8}.
\end{theorem}

As in ~\cite{BF}, for sufficiently small $\e$ this allows us to derive
the asymptotics for specific eigenvalues, and we give the explicit
expansion for the first eigenvalue in terms of the functions $H$ and $h_{-}$
in the case where $H_{2}$ is negative for $x'\neq\xm$.

\begin{theorem}\label{th2.2}
For any $N\geqslant 1$ there exists $\e_0=\e_0(N)$ such that for
$\e\leqslant \e_0$ the first $N$ eigenvalues of $-\D^D_{\Om_\e}$
are $\l_{1,m}(\e)$, $m=1,\ldots,N$. If
\begin{equation}\label{2.4}
k=1,\quad
H_2(x')=-\frac{1}{2}\sum\limits_{i=1}^{d-1}\a_i^2x_i^2,
\end{equation}
the lowest eigenvalue $\l_{1,1}(\e)$ has the asymptotic
expansion
\begin{align}
&\l_{1,1}(\e)=\frac{c_0^{(1,1)}}{\e^2}+\frac{c_2^{(1,1)}}{\e}
+c_4^{(1,1)}+\Odr(\e^{1/2}),\quad \e\to+0, \label{2.5}
\\
&c_0^{(1,1)}=\frac{\pi^2}{H_0^2},\quad
c_2^{(1,1)}=\sum\limits_{j=1}^{d-1}\tht_j,\quad
\tht_j:=\frac{\pi\a_j}{H_0^{3/2}},\nonumber
\\
&c_4^{(1,1)}=\frac{\pi^2}{H_0^4}\big((3H_2^2(\xi')-2H_0H_4(\xi')
)\Psi_0,\Psi_0\big)_{L_2(\mathbb{R}^{d-1})}\nonumber
\\
&\hphantom{c_4^{(1,1)}=}+\frac{\pi^2}{H_0^2}\sum\limits_{i=1}^{d-1}
\left( \frac{\p h_1}{\p x_i}\right)^2-\frac{2\pi^2}{H_0^3}
\big(H_3(\xi')\widetilde{\Psi}_1,\Psi_0\big)_{L_2(\mathbb{R}^{d-1})},
\nonumber
\\
&\Psi_0(\xi'):=\prod\limits_{j=1}^{d-1}
\frac{\tht_j^{1/4}}{\pi^{1/4}}\E^{-\frac{\tht_j\xi_j^2}{2}},\label{2.8}
\\
&
\widetilde{\Psi}_1(\xi'):=\Psi_0(\xi')\left(\sum\limits_{p,j=1}^{d-1}
\frac{3\pi^2\b_{ppj}\xi_j}{2H_0^3\tht_j (2\tht_p+\tht_j)}\right.
\nonumber
\\
&\hphantom{\widetilde{\Psi}_1(\xi'):=\Psi_0(\xi')\Bigg(}\left.
-\sum\limits_{p,q,j=1}^{d-1}
\frac{\pi^2\b_{pqj}\xi_p\xi_q\xi_j}{H_0^3(\tht_p+\tht_q+\tht_j)}
\right),\nonumber
\end{align}
where it is assumed that $H_3(x')$ is written as
\begin{equation*}
H_3(x')=\sum\limits_{p,q,j=1}^{d-1} \b_{pqj}\xi_p\xi_q\xi_j,
\end{equation*}
and the constants $\b_{pqj}$ are invariant under each
permutation of the indices $p$, $q$, $j$:
\begin{equation}\label{2.7}
\b_{pqj}=\b_{pjq}=\b_{qpj}=\b_{qjp}=\b_{jpq}=\b_{jqp}.
\end{equation}
\end{theorem}

\begin{rmk}\label{rm2.1}
The assumption (\ref{2.4}) for $H_2$ is not a restriction, since
we can always achieve such form for $H_2$ by an appropriate
change of variables.
\end{rmk}

\section{Proof of Theorem~\ref{th2.1}\label{proofthm1}}

In this section we construct the asymptotics for the eigenvalues
and the eigenfunctions of the operator $-\D^{D}_{\Om_\e}$. This is
first done formally, and justified rigorously afterwards.
In the formal construction we employ the same approach as was used
in \cite[Sec. 3]{BF}.

We are going to construct formally the asymptotic expansions for
the eigenvalues $\l_{n,m}(\e)$, $m=M,\ldots,M+N-1$ which we relabel
as $\l_\e^{(m)}$, $m=1,\ldots,N$. We denote the associated
eigenfunctions by $\psi_\e^{(m)}$. We construct their asymptotic
expansions as the series
\begin{align}
&
\begin{aligned}
&\l_\e^{(m)}=\e^{-2}\mu_\e^{(m)},\quad
\mu_\e^{(m)}=c_0^{(m)}+\sum\limits_{i=2k}^{\infty}
c_i^{(m)}\eta^i,\\
&\psi_\e^{(m)}(x)=\sqrt{H(x')}\widetilde{\psi}_\e^{(m)}(x),\quad 
\widetilde{\psi}_\e^{(m)}(x)=
\sum\limits_{i=0}^{\infty}\eta^i\psi_i^{(m)}(\xi),
\end{aligned}
\label{3.1}
\\
&\xi=(\xi',\xi_d),\quad\xi':=\frac{x'-\xm}{\eta},\quad
\xi_d:=\frac{x_d+\e h_-(x')}{\e H(x')}.\nonumber
\end{align}
We postulate the functions $\psi_i^{(m)}(\xi)$ to be
exponentially decaying as $\xi'\to+\infty$. It means that they
are exponentially small outside $\Om_\e^\d$ (with respect to
$\e$). In terms of the variables $\xi$ the domain $\Om_\e^\d$
becomes $\{\xi: |\xi'|<\d\eta^{-1}, 0<\xi_d<1\}$. As $\eta\to0$,
it ``tends'' to the layer $\Pi:=\{\xi: 0<\xi_d<1\}$ and this is
why we shall construct the functions $\psi_i$ as defined on
$\Pi$.

We rewrite the eigenvalue equation for $\psi_\e$ and $\l_\e$ in
the variables $\xi$,
\begin{equation}\label{3.2}
\begin{gathered}
\begin{aligned}
-\Bigg[&\eta^{2k}\D_{\xi'}+K_{d}\frac{\p^2}{\p\xi_d^2}
+\sum\limits_{i=1}^{d-1} \eta^{2k+1} \left( \frac{\p}{\p\xi_i}K_{i}
\frac{\p}{\p\xi_d} +\frac{\p}{\p\xi_d} K_{i}
\frac{\p}{\p\xi_i}\right)
\\
&+ \eta^{2k+2}\sum\limits_{i=1}^{d-1} \frac{\p}{\p\xi_d} K_{i}^2
\frac{\p}{\p\xi_d} +\eta^{2k+2} K_0
\Bigg]\widetilde{\psi}_\e^{(m)}=\mu_\e^{(m)}\widetilde{\psi}_\e^{(m)} 
\quad \text{in}\quad \Pi,
\end{aligned}
\\
\psi_\e^{(m)}=0\quad\text{on}\quad\p\Pi,
\end{gathered}
\end{equation}
where $K_{i}=K_{i}(\xi,\eta)$, $i=0,\ldots,d$,
\begin{align*}
&K_d(\xi,\eta)=\frac{1}{H^2(\xm+\eta\xi')},
\\
&K_i(\xi,\eta)=\frac{1}{H(\xm+\eta\xi')} \left[\frac{\p h_-} {\p
x_i}(\xm+\eta\xi')-\xi_d\frac{\p H}{\p x_i}(\xm+\eta\xi')\right],
\\
&K_0(\xi,\eta)=\frac{1}{2} H^{-1}(\xm+\eta\xi')\D_{x'}
H(\xm+\eta\xi')
\\
&\hphantom{K_0(\xi,\eta)=}
-\frac{1}{4}H^{-2}(\xm+\eta\xi')\left|\nabla_{x'}
H(\xm+\eta\xi')\right|^2.
\end{align*}

\begin{rmk}\label{rm3.1b}
We have introduced the factor $\sqrt{H(x')}$ in the series
(\ref{3.1}) for $\psi_\e^{(m)}$ in order to have a symmetric
differential operator in the equation (\ref{3.2}).
\end{rmk}

We expand the functions $K_i$ into the Taylor series w.r.t. $\eta$
and employ (\ref{2.1}) to obtain
\begin{equation}
\begin{aligned}
&K_d(\xi,\eta)=H_0^{-2}+\sum\limits_{j=2k}^{\infty} \eta^j
P_j^{(d)}(\xi'),
\\
&K_{i}(\xi,\eta)=\sum\limits_{j=0}^{\infty}\eta^j
K_j^{(i)}(\xi),
\\
&K_j^{(i)}(\xi):=P_j^{(i)}(\xi')+\xi_d Q_j^{(i)}(\xi'),\quad
i=1,\ldots,d-1,
\\
&K_0(\xi,\eta)=\sum\limits_{i=0}^\infty \eta^i P_i^{(0)}(\xi'),
\end{aligned}\label{3.3}
\end{equation}
where $P_j^{(i)}$, $Q_j^{(i)}$ are polynomials, and, in
particular,
\begin{equation}\label{3.4}
\begin{aligned}
& P_{2k}^{(d)}(\xi')=-\frac{2 H_{2k}(\xi')}{H_0^3}, &&
P_{2k+1}^{(d)}(\xi')=-\frac{2 H_{2k+1}(\xi')}{H_0^3},
\\
& P_0^{(i)}(\xi')=\frac{1}{H_0} \frac{\p h_1}{\p x_i}(\xm), &&
Q_0^{(i)}(\xi')=0.
\end{aligned}
\end{equation}
We substitute (\ref{3.1}), (\ref{3.3}), (\ref{3.4}) into
(\ref{3.2}) and equate the coefficients of like powers of
$\eta$. This leads us to the following boundary value problems for $\psi_i^{(m)}$,
\begin{align}
&
\begin{aligned}
&\left(\frac{1}{H_0^2}\frac{\p^2}{\p\xi_d^2}+c_0^{(m)}\right)
\psi_j^{(m)}=0 \quad \text{in}\quad \Pi,
\\
&\hphantom{\Big(}\psi_j^{(m)}=0\quad\text{on}\quad\p\Pi,\quad
j=0,\ldots,2k-1,
\end{aligned}
\label{3.5}
\\
&
\begin{aligned}
-&\left(\frac{1}{H_0^2}\frac{\p^2}{\p\xi_d^2}+c_0^{(m)}\right)
\psi_{2k}^{(m)}
\\
&\hphantom{\frac{1}{H_0^2}\frac{\p^2}{\p\xi_d^2}+c_0^{(m)}}=\left(
\D_{\xi'}-\frac{2H_{2k}(\xi')}{H_0^3}\frac{\p^2}{\p\xi_d^2}
+c_{2k}^{(m)}\right) \psi_0^{(m)}\quad \text{in}\quad \Pi,
\\
&\hphantom{\Big(}\psi_{2k}^{(m)}=0\quad\text{on}\quad\p\Pi,
\end{aligned}\label{3.6}
\\
&
\begin{aligned}
-&\left(\frac{1}{H_0^2}\frac{\p^2}{\p\xi_d^2}+c_0^{(m)}\right)
\psi_j^{(m)}=\left(
\D_{\xi'}-\frac{2H_{2k}(\xi')}{H_0^3}\frac{\p^2}{\p\xi_d^2}
+c_{2k}^{(m)} \right) \psi_{j-2k}^{(m)}
\\
&\hphantom{\Bigg(\frac{1}{H_0^2}\frac{\p^2}{\p\xi_d^2}+c_0^{(m)}}
+c_j^{(m)}\psi_0^{(m)}+\sum\limits_{q=1}^{j-2k-1}
c_{j-q}^{(m)}\psi_q^{(m)}+F_j^{(m)}\quad \text{in}\quad \Pi,
\\
&\hphantom{\Big(}\psi_j^{(m)}=0\quad\text{on}\quad\p\Pi,\quad
j\geqslant 2k+1,
\end{aligned}\label{3.7}
\\
&
\begin{aligned}
&F_j^{(m)}:=\sum\limits_{q=0}^{j-2k-1}
\mathcal{L}_{j-q-2k}\psi_q^{(m)},
\\
&\mathcal{L}_j:=\sum\limits_{i=1}^{d-1} \left( \frac{\p}{\p\xi_d}
K_{j-1}^{(i)} \frac{\p}{\p\xi_i} + \frac{\p}{\p\xi_i} K_{j-1}^{(i)}
\frac{\p}{\p\xi_d}\right)
\\
&\hphantom{\mathcal{L}_j:=}+ \sum\limits_{i=1}^{d-1}
\sum\limits_{s=0}^{j-2} \frac{\p}{\p\xi_d} K_s^{(i)}
K_{j-s-2}^{(i)}\frac{\p}{\p\xi_d}+
P_{j+2k}^{(d)}\frac{\p^2}{\p\xi_d^2}+P^{(0)}_{j-2},
\end{aligned}\label{3.8}
\end{align}
where $P_{-1}^{(0)}=0$.  Problems (\ref{3.5}) can be solved
explicitly with
\begin{equation}\label{3.9}
\psi_j^{(m)}(\xi)=\Psi_j^{(m)}(\xi')\sin\pi n\xi_d,\quad
c_0=\frac{\pi^2 n^2}{H_0^2},
\end{equation}
where $j=0,\ldots, 2k-1$, and $\Psi_j^{(m)}$ are the functions to be
determined. The last identity proves formula (\ref{2.3}) for
$c_0^{(n,m)}$.

We consider the problem (\ref{3.6}) as posed on the interval
$(0,1)$ and depending on $\xi'$. It is solvable, if and only if
the right-hand side is orthogonal to $\sin\pi n \xi_d$ in
$L_2(0,1)$. It implies the equation
\begin{equation}\label{3.10}
-\left(\D_{\xi'}+\frac{2\pi^2 n^2 H_{2k}(\xi')}{H_0^3}\right)
\Psi_0^{(m)}=c_{2k}^{(m)}\Psi_0^{(m)}\quad \text{in}\quad
\mathbb{R}^{d-1}.
\end{equation}
Thus, $c_{2k}^{(m)}$ is an eigenvalue of the operator $G_n$, i.e.,
$c_{2k}^{(m)}=\L$. Then $\Psi_0^{(m)}$ is one of the eigenfunctions
associated with $\L$. These eigenfunctions are assumed to be
orthonormalized in $L_2(\mathbb{R}^{d-1})$. We substitute the equation
(\ref{3.10}) into (\ref{3.6}) and see that the formula (\ref{3.9})
is valid also for $j=2k$.

The problems (\ref{3.7}) are solvable, if and only if the
right-hand sides are orthogonal to $\sin\pi n\xi_d$ in
$L_2(0,1)$. It gives rise to the equations
\begin{align}
&(G_n-\L)\Psi_{j-2k}^{(m)}=f_j^{(m)}+\sum\limits_{q=1}^{j-2k-1}
c_{j-q}^{(m)}\Psi_q^{(m)}+c_j^{(m)}\Psi_0^{(m)},\label{3.11a}
\\
&
\begin{aligned}
&\Psi_j^{(m)}=\Psi_j^{(m)}(\xi'):=2\int\limits_0^1
\psi_j^{(m)}(\xi) \sin\pi n\xi_d\di\xi_d,
\\
&f_j^{(m)}=f_j^{(m)}(\xi'):=2\int\limits_0^1 F_j^{(m)}(\xi)
\sin\pi n\xi_d\di\xi_d.
\end{aligned}
\label{3.11b}
\end{align}

To solve the problems (\ref{3.7}), (\ref{3.11a}) we need some
auxiliary lemmata. The first of these follows from standard results
in the spectral theory of self-adjoint operators.

\begin{lemma}\label{lm3.1}
Let $f\in L_2(\mathbb{R}^{d-1})$. The equation
\begin{equation}\label{3.12}
(G_n-\L)u=f
\end{equation}
is solvable, if and only if
\begin{equation*}
(f,\Psi_0^{(m)})_{L_2(\mathbb{R}^{d-1})}=0,\quad m=1,\ldots,N.
\end{equation*}
The solution is unique up to a linear combination of the
functions $\Psi_0^{(m)}$.
\end{lemma}

By $\mathfrak{g}_n$ we denote the sesquilinear form associated
with $G_n$,
\begin{equation*}
\mathfrak{g}_n[u,v]=(\nabla u,\nabla v)_{L_2(\mathbb{R}^{d-1})}
-(H_{2k}u,v)_{L_2(\mathbb{R}^{d-1})}.
\end{equation*}
The domain of this form is
\begin{equation*}
\Dom(\mathfrak{g}_n)=\H^1(\mathbb{R}^{d-1})\cap \{u:
(1+|\xi'|^k)u\in L_2(\mathbb{R}^{d-1})\}.
\end{equation*}
By $\Dom(G_n)$ we denote the domain of $G_n$. The set
$C_0^\infty(\mathbb{R}^{d-1})$ is dense in
$\Dom(\mathfrak{g}_n)$ in the topology induced by
$\mathfrak{g}_n$ (\cite[Ths. 1.8.1, 1.8.2]{D}).

\begin{lemma}\label{lm3.2}
Let $f\in L_2(\mathbb{R}^{d-1})$, $u\in
L_2(\mathbb{R}^{d-1})\cap \H^1(S)$ for each bounded domain
$S\subset \mathbb{R}^{d-1}$ and for each $\phi\in
\Dom(\mathfrak{g})$ the identity
\begin{equation}\label{3.14}
\int\limits_{\mathbb{R}^{d-1}} \nabla u\cdot \nabla\phi\di\xi'-
\int\limits_{\mathbb{R}^{d-1}}
(H_{2k}(\xi')-\L)u\overline{\phi}\di\xi'=
\int\limits_{\mathbb{R}^{d-1}} f\overline{\phi}\di\xi'
\end{equation}
holds true. Then $u\in\Dom(G_n)$ and the equation (\ref{3.12})
is valid.
\end{lemma}
\begin{proof}
Let $\chi_1=\chi_1(t)$ be a non-negative infinitely
differentiable cut-off function taking values in $[0,1]$,
equalling one as $t<1$, and vanishing as $t>2$. It is clear that
for each $t>0$ the function $u(\xi')\chi_1(|\xi'|t)$ belongs to
$\Dom(\mathfrak{g}_n)$. We substitute
$\phi(\xi')=u(\xi')\chi_1(|\xi'|t)$ into (\ref{3.14}) and
integrate by parts,
\begin{equation}\label{3.15}
\begin{aligned}
 \|\chi_1\nabla u\|_{L_2(\mathbb{R}^{d-1})}^2 &-
(H_{2k}\chi_1 u,\chi_1u)_{L_2(\mathbb{R}^{d-1})}=(\chi_1
f,\chi_1 u)_{L_2(\mathbb{R}^{d-1})}\\ &+\frac{1}{2}
(u\D_{\xi'}\chi_1^2,u)_{L_2(\mathbb{R}^{d-1})}
 +\L \|\chi_1
u\|_{L_2(\mathbb{R}^{d-1})}^2.
\end{aligned}
\end{equation}
Hence,
\begin{equation}\label{3.16}
\begin{aligned}
\|\nabla u\|^2_{L_2(\mathcal{B}_{t^{-1}}'(0))}
&-(H_{2k}u,u)_{L_2(\mathcal{B}_{t^{-1}}'(0))} \leqslant
\|f\|_{L_2(\mathbb{R}^{d-1})}\|u\|_{L_2(\mathbb{R}^{d-1})}
\\
&+ Ct^2\|u\|_{L_2(\mathbb{R}^{d-1})}
+\L\|u\|_{L_2(\mathbb{R}^{d-1})}^2,
\end{aligned}
\end{equation}
where the constant $C$ is independent of $t$, and
$\mathcal{B}'_{r}(a):=\{\xi': |\xi'-a|<r\}$. Passing to the
limit as $t\to+0$, we conclude that $u\in\Dom(\mathfrak{g}_n)$
and in view of (\ref{3.14}) this function belongs to
$\Dom(\mathfrak{g}_n)$ and solves the equation (\ref{3.12}).
\end{proof}

Let $\mathfrak{V}$ be the space of the functions $f\in
C^\infty(\mathbb{R}^{d-1})$ such that
\begin{equation*}
(1+|\xi'|^\g)\frac{\p^\tau f}{\p{\xi'}^\tau}\in
L_2(\mathbb{R}^{d-1})
\end{equation*}
for each $\tau\in \mathbb{Z}^d_+$, $\g\in\mathbb{Z}_+$.

\begin{lemma}\label{lm3.3}
Let $f\in \mathfrak{V}$, and $u$ be a solution to (\ref{3.12}).
Then $u\in \mathfrak{V}$.
\end{lemma}

\begin{proof}
Since $u\in\Dom(G_n)$, we have $\nabla u\in L_2(\mathbb{R}^{d-1})$,
$(1+|\xi'|^k)u\in L_2(\mathbb{R}^{d-1})$, and due to standard
smoothness improving theorems $u\in C^\infty(\mathbb{R}^{d-1})$. The
identity (\ref{3.15}) is also valid with $\chi_1$ replaced by
$\chi_1(|\xi'|t)|\xi'|^\b$. Employing this identity and proceeding
as in (\ref{3.16}), we check that $(1+|\xi'|^\b)\nabla u\in
L_2(\mathbb{R}^{d-1})$, $(1+|\xi'|^{k+\b}) u\in
L_2(\mathbb{R}^{d-1})$, if $(1+|\xi'|^\b) u\in
L_2(\mathbb{R}^{d-1})$ for some $\b\in\mathbb{Z}_+$. Applying this
fact  by induction and using that $(1+|\xi'|^k) u\in
L_2(\mathbb{R}^{d-1})$, we conclude that $(1+|\xi'|^\g)\nabla u\in
L_2(\mathbb{R}^{d-1})$, $(1+|\xi'|^{k+\g}) u\in
L_2(\mathbb{R}^{d-1})$ for each $\g\in \mathbb{Z}_+$.

We differentiate the equation (\ref{3.12}) w.r.t. $\xi_i$,
\begin{equation*}
(G_n-\L)\frac{\p u}{\p\xi_i}=\frac{\p f}{\p\xi_i}+\frac{\p
H_{2k}}{\p\xi_i}u.
\end{equation*}
The right hand side belongs to $L_2(\mathbb{R}^{d-1})$ and the
function $\frac{\p u}{\p\xi_i}$ satisfies the hypothesis of
Lemma~\ref{lm3.2}. Applying this lemma, we see that $\frac{\p u
}{\p\xi_i}\in \Dom(G_n)$. Proceeding as above, one can make sure
that
\begin{equation*}
(1+|\xi'|^\g)\nabla\frac{\p u}{\p\xi_i}\in L_2(\mathbb{R}^{d-1})
\end{equation*}
for each $\g\in \mathbb{Z}_+$. Repeating the described process,
we complete the proof.
\end{proof}

As it follows from Lemma~\ref{lm3.1}, the solvability condition
of the equation (\ref{3.11a}) is
\begin{equation*}
(f_j^{(m)},\Psi_0^{(l)})_{L_2(\mathbb{R}^{d-1})}+\sum\limits_{q=1}^{j-2k-1}
c_{j-q}^{(m)}(\Psi_q^{(m)},\Psi_0^{(l)})_{L_2(\mathbb{R}^{d-1})}+c_j^{(m)}\d_{ml}=0,
\end{equation*}
where $m,l=1,\ldots,N$, and $\d_{ml}$ is the Kronecker delta.
Here we have supposed that the functions $\Psi_0^{(m)}$ are
orthonormalized in $L_2(\mathbb{R}^{d-1})$. In view of
(\ref{3.11b}) these identities can be rewritten as
\begin{equation}\label{3.10b}
2(F_j^{(m)},\psi_0^{(l)})_{L_2(\Pi)}+
2\sum\limits_{q=1}^{j-2k-1}c_{j-q}^{(m)}
(\psi_q^{(m)},\psi_0^{(l)})_{L_2(\Pi)}+c_j^{(m)}\d_{ml}=0,
\end{equation}
where $m,l=1,\ldots,N$.

Consider the problem (\ref{3.7}) for $j=2k+1$. The solvability
condition is the equation (\ref{3.11a}) for the same $j$. Since
$\Psi_0^{(m)}\in \mathfrak{V}$, the same is true for
$f_{2k+1}^{(m)}$. By (\ref{3.10b}), this equation is solvable, if
and only if
\begin{align}
&T_{ml}^{(2k+1)}+c_{2k+1}^{(m)}\d_{ml}=0,\quad m,l=1,\ldots,N,
\label{3.18}
\\
&T_{ml}^{(2k+1)}:=2\left(\mathcal{L}_1
\psi_0^{(m)},\psi_0^{(l)}\right)_{L_2(\Pi)}.\nonumber
\end{align}
The definition of $\mathcal{L}_1$ and (\ref{3.4}) yield
\begin{equation}\label{3.21a}
T_{ml}^{(2k+1)}=2\pi^2 n^2 H_0^{-3}(H_{2k+1}\Psi_0^{(m)},
\Psi_0^{(l)})_{L_2(\mathbb{R}^{d-1})}.
\end{equation}
Hence, the matrix $T^{(2k+1)}$ with the entries $T_{ml}^{(2k+1)}$ is
symmetric. This matrix describes a quadratic form on the space
spanned over $\Psi_0^{(m)}$, $m=1,\ldots,N$. By the theorem on the
simultaneous diagonalization of two quadratic forms we conclude that
the eigenfunctions $\Psi_0^{(m)}$ can be chosen as orthonormalized
in $L_2(\mathbb{R}^{d-1})$ and, in addition, so that the matrix
$T^{(2k+1)}$ is diagonal. In what follows we assume that these
functions are chosen in such a way. Then identities (\ref{3.18})
imply
\begin{equation}\label{3.19}
c_{2k+1}^{(m)}=-\tau_m^{(2k+1)},
\end{equation}
where $\tau_m^{(2k+1)}$ are the eigenvalues of $T^{(2k+1)}$.

By Lemma~\ref{lm3.1} the solution to (\ref{3.11a}) for $j=2k+1$
reads as follows
\begin{equation}\label{3.21}
\Psi_1^{(m)}(\xi')=\Phi_1^{(m)}(\xi')+\sum\limits_{p=1}^{N}
b_{p,1}^{(m)}\Psi_0^{(p)},
\end{equation}
where $\Phi_1^{(m)}$ is orthogonal to all $\Psi_0^{(l)}$,
$l=1,\ldots,N$, in $L_2(\mathbb{R}^{d-1})$ and $b_{p,1}^{(m)}$ are
constants to be found. It follows from Lemma~\ref{lm3.2} that
$\Phi_1^{(m)}\in\mathfrak{V}$. The definition (\ref{3.8}) of
$\mathcal{L}_1$ and the equation (\ref{3.11a}) for $j=2k+1$ imply
that the right-hand side of the equation in (\ref{3.7}) for $j=2k+1$
is zero. Hence, the solution to the problem (\ref{3.7}) for $j=2k+1$
is given by the formula (\ref{3.9}), where $\Psi^{(m)}_{2k+1}$ is to
be found. We substitute (\ref{3.9}), (\ref{3.21}) into the equation
(\ref{3.11a}) for $j=2k+2$. In view of (\ref{3.10b}) and
(\ref{3.19}) the solvability condition for this equation is as
follows
\begin{equation}\label{3.24}
\begin{aligned}
&b_{l,1}^{(m)}(\tau_l^{(2k+1)}-\tau_m^{(2k+1)})
+c_{2k+2}^{(m)}\d_{ml}
\\
&+2 \big(\mathcal{L}_2\psi_0^{(m)}
+\mathcal{L}_1\Phi_1^{(m)}\sin\pi
n\xi_d,\psi_0^{(l)}\big)_{L_2(\Pi)}=0,\quad l=1,\ldots,N.
\end{aligned}
\end{equation}
Assume that all the eigenvalues $\tau_m^{(2k+1)}$ are different.
In this case the last identities imply
\begin{equation}\label{3.25}
\begin{aligned}
&b_{l,1}^{(m)}=\frac{2\big(\mathcal{L}_2\psi_0^{(m)}
+\mathcal{L}_1\Phi_1^{(m)}\sin\pi
n\xi_d,\psi_0^{(l)}\big)_{L_2(\Pi)}}
{\tau_m^{(2k+1)}-\tau_l^{(2k+1)}},\quad m\not=l,
\\
&c_{2k+2}^{(m)}=-2\big(\mathcal{L}_2\psi_0^{(m)}
+\mathcal{L}_1\Phi_1^{(m)}\sin\pi
n\xi_d,\psi_0^{(m)}\big)_{L_2(\Pi)},
\end{aligned}
\end{equation}
and we can also let $b_{m,1}^{(m)}=0$.

Now suppose that all the eigenvalues $\tau_m^{(2k+1)}$ are
equal. In this case the equations (\ref{3.24}) do not allow us
to determine the constants $b_{l,1}^{(m)}$ for $m\not=l$.
Consider the matrix $T^{(2k+2)}$ with the entries
\begin{equation*}
T^{(2k+2)}_{ml}:=2\big(\mathcal{L}_2\psi_0^{(m)}
+\mathcal{L}_1\Phi_1^{(m)}\sin\pi
n\xi_d,\psi_0^{(l)}\big)_{L_2(\Pi)}.
\end{equation*}

\begin{lemma}\label{lm3.4}
The matrix $T^{(2k+2)}$ is symmetric.
\end{lemma}
\begin{proof}
Integrating by parts, we obtain
\begin{equation*}
T^{(2k+2)}_{ml}=2\big(\psi_0^{(m)},
\mathcal{L}_2\psi_0^{(l)}\big)_{L_2(\Pi)}
+2\big(\Phi_1^{(m)}\sin\pi
n\xi_d,\mathcal{L}_1\psi_0^{(l)}\big)_{L_2(\Pi)}.
\end{equation*}
Since by (\ref{3.7})
\begin{equation*}
\mathcal{L}_1\psi_0^{(l)}=
-\left(\frac{1}{H_0^2}\frac{\p^2}{\p\xi_d^2}+ c_0^{(m)}\right)
\psi_{2k+1}^{(m)}- c_{2k+1}^{(m)}\psi_0^{(m)},
\end{equation*}
in view of (\ref{3.9}), (\ref{3.11a}), (\ref{3.21}) we have
\begin{align*}
2\big(&\Phi_1^{(m)}\sin\pi
n\xi_d,\mathcal{L}_1\psi_0^{(l)}\big)_{L_2(\Pi)}=\big(\Phi_1^{(m)},
(G_n-\L)\Phi_1^{(l)}\big)_{L_2(\mathbb{R}^{d-1})}
\\
&= \big((G_n-\L)\Phi_1^{(m)},
\Phi_1^{(l)}\big)_{L_2(\mathbb{R}^{d-1})}=
2\big(\mathcal{L}_1\psi_0^{(m)}, \Phi_1^{(m)}\sin\pi
n\xi_d\big)_{L_2(\Pi)}.
\end{align*}
\end{proof}

Since we supposed that all the eigenvalues of $T^{(2k+1)}$ are
equal, we can make orthogonal transformation in the space
spanned over $\Psi_0^{(m)}$, $m=1,\ldots,N$, without destroying
the orthonormality in $L_2(\Pi)$ and diagonalization of
$T^{(2k+1)}$. We employ this freedom to diagonalize the matrix
$T^{(2k+2)}$ which is possible due to Lemma~\ref{lm3.4}. After
such diagonalization we see that the coefficients $c_{2k+2}^{(m)}$
are determined by the eigenvalues of the matrix $T^{(2k+2)}$:
\begin{equation*}
c_{2k+2}^{(m)}=-\tau^{(2k+2)}_m.
\end{equation*}
If all these eigenvalues are distinct, we can determine the
numbers $b_{l,1}^{(m)}$ at the next step by formulae similar
to (\ref{3.25}). If all these eigenvalues are identical, at the next
step we should consider the next matrix $T^{(2k+3)}$ and
diagonalize it.

There exists one more possibility. Namely, the matrix
$T^{(2k+1)}$ can have different multiple eigenvalues. We do not
treat this case here. The reason is that the formal construction
of the asymptotics is rather complicated from the technical
point of view and at the same time it does not require any new
ideas in comparison with the cases discussed above. Thus,
from now on, we consider two cases only. More precisely, in the first case
we assume that the matrix $T^{(2k+1)}$ has $N$ different
eigenvalues $\tau^{(2k+1)}_m$, $m=1,\ldots,N$. In the second
case we suppose that the matrix $T^{(2k+1)}$ has only one
eigenvalue $\tau^{(2k+1)}$ with multiplicity $N$, while the
matrix $T^{(2k+2)}$ has $N$ different eigenvalues
$\tau^{(2k+2)}_m$, $m=1,\ldots,N$.

\begin{lemma}\label{lm3.7} Assume that the
matrix $T^{(2k+1)}$ has $N$ different eigenvalues and choose
$\Psi_0^{(m)}$ being orthonormalized in $L_2(\mathbb{R}^{d-1})$
and so that the matrix $T^{(2k+1)}$ is diagonal.  Then the
problems (\ref{3.5}), (\ref{3.6}), (\ref{3.7}) have solutions
\begin{equation*}
\psi_j^{(m)}(\xi)=\widetilde{\psi}_j^{(m)}(\xi)+
\widetilde{\Psi}_j^{(m)}(\xi')\sin\pi n\xi_d+
\sum\limits_{p=1}^{N} b_{j,p}^{(m)}\psi_0^{(p)}(\xi).
\end{equation*}
Here the functions $\widetilde{\psi}_j^{(m)}$ are zero for
$j\leqslant 2k+1$, while for other $j$ they solve the problems
\begin{align*}
-&\left(\frac{1}{H_0^2}\frac{\p^2}{\p\xi_d^2}+c_0^{(m)}\right)
\widetilde{\psi}_j^{(m)}=\left(
\D_{\xi'}-\frac{2H_{2k}(\xi')}{H_0^3}\frac{\p^2}{\p\xi_d^2}+\L\right)
\widetilde{\psi}_{j-2k}^{(m)}
\\
&+\sum\limits_{q=2k+2}^{j-2k-1}c^{(m)}_{j-q}\widetilde{\psi}^{(m)}_q
+F_j^{(m)}-2(F_j^{(m)},\sin\pi n\xi_d)_{L_2(0,1)}\sin\pi
n\xi_d\quad \text{in}\quad \Pi,
\\
&\hphantom{\Big(}\widetilde{\psi}_j^{(m)}=0\quad\text{on}\quad\p\Pi,
\end{align*}
and are represented as finite sums
\begin{equation*}
\widetilde{\psi}_j^{(m)}(\xi)=\sum\limits_\vs
\psi_{j,\vs,1}^{(m)}(\xi')\psi_{j,\vs,2}^{(m)}(\xi_d),
\end{equation*}
where $\psi_{j,\vs,1}^{(m)}\in\mathfrak{V}$,
$\psi_{j,\vs,2}^{(m)}\in C_0^\infty[0,1]$,
$\psi_{j,\vs,2}^{(m)}(0)=\psi_{j,\vs,2}^{(m)}(1)=0$, and the
functions $\psi_{j,\vs,2}^{(m)}$ are orthogonal to $\sin\pi
n\xi_d$ in $L_2(0,1)$. The functions
$\widetilde{\Psi}_j^{(m)}\in \mathfrak{V}$ are solutions to the
equations (\ref{3.11a}) and are orthogonal to all
$\Psi_0^{(l)}$, $l=1,\ldots,N$, in $L_2(\mathbb{R}^{d-1})$. The
constants $b^{(m)}_{j,p}$ and $c^{(m)}_j$ are determined by the
formulae
\begin{align*}
&
\begin{aligned}
&b^{(m)}_{0,l}=\d_{ml}, \quad b_{j,m}^{(m)}=0,\quad j\geqslant
1,
\\
&b^{(m)}_{j,l}=\frac{
2(\widetilde{F}_{j+2k+1}^{(m)},\psi_0^{(l)})
+\sum\limits_{q=1}^{j-1} c^{(m)}_{j+2k-q+1} b_{q,l}^{(m)}
}{\tau_{2k+1}^{(m)}-\tau_{2k+1}^{(l)}},
\quad m\not=l,\quad j\geqslant 1,
\end{aligned}
\\
&c^{(m)}_{2k}=\L,\quad c^{(m)}_{2k+1}=-\tau^{(2k+1)}_m,
\\
&c^{(m)}_j=-2(\widetilde{F}_j^{(m)},\psi_0^{(m)})_{L_2(\Pi)},
\quad j\geqslant 2k+2,
\\
&\widetilde{F}^{(m)}_j=\sum\limits_{q=0}^{j-2k-1}
\mathcal{L}_{j-q-2k}(\widetilde{\psi}^{(m)}_q +
\widetilde{\Psi}^{(m)}_q \sin\pi n\xi_d)+
\sum\limits_{q=0}^{j-2k-2}\sum\limits_{p=1}^{N} b^{(m)}_{q,p}
\mathcal{L}_{j-q-2k} \psi^{(p)}_0.
\end{align*}
\end{lemma}

\begin{lemma}\label{lm3.8}
Assume that all the eigenvalues of the matrix $T^{(2k+1)}$ are
identical and that the matrix $T^{(2k+2)}$ has $N$ different eigenvalues,
and choose $\Psi_0^{(m)}$ being orthonormalized in
$L_2(\mathbb{R}^{d-1})$ so that the matrices $T^{(2k+1)}$
and $T^{(2k+2)}$ are diagonal. Then problems (\ref{3.5}),
(\ref{3.6}), (\ref{3.7}) have solutions
\begin{align*}
\psi_j^{(m)}(\xi)=&\widetilde{\psi}_j^{(m)}(\xi)+
\widetilde{\Psi}_j^{(m)}(\xi')\sin\pi n\xi_d
\\
&+ \sum\limits_{p=1}^{N}
b_{j-1,p}^{(m)}\Phi_1^{(p)}(\xi')\sin\pi n\xi_d+
\sum\limits_{p=1}^{N} b_{j,p}^{(m)}\psi_0^{(p)}(\xi).
\end{align*}
Here the functions $\widetilde{\psi}_j^{(m)}$ are zero for
$j\leqslant 2k+1$, while for other $j$ they solve the problems
\begin{align*}
-&\left(\frac{1}{H_0^2}\frac{\p^2}{\p\xi_d^2}+c_0^{(m)}\right)
\widetilde{\psi}_j^{(m)}=\left(
\D_{\xi'}-\frac{2H_{2k}(\xi')}{H_0^3}\frac{\p^2}{\p\xi_d^2}+\L\right)
\widetilde{\psi}_{j-2k}^{(m)}
\\
&+\sum\limits_{q=2k+2}^{j-2k-1}c^{(m)}_{j-q}\widetilde{\psi}^{(m)}_q
+\widetilde{F}_j^{(m)}-2(\widetilde{F}_j^{(m)},\sin\pi
n\xi_d)_{L_2(\Pi)}\sin\pi n\xi_d\quad \text{in}\quad \Pi,
\\
&\hphantom{\Big(}\widetilde{\psi}_j^{(m)}=0\quad\text{on}\quad\p\Pi,
\\
&\widetilde{F}^{(m)}_j:=\sum\limits_{q=0}^{j-2k-1}
\mathcal{L}_{j-q-2k} (\widetilde{\psi}^{(m)}_q
+\widetilde{\Psi}^{(m)}_q\sin\pi n\xi_d)
\\
&\hphantom{\widetilde{F}^{(m)}_j:=}+\sum\limits_{p=1}^{N}
\sum\limits_{q=1}^{j-2k-2} b^{(m)}_{q-1,p} \mathcal{L}_{j-q-2k}
\Phi_1^{(p)}\sin\pi n\xi_d
\\
&\hphantom{\widetilde{F}^{(m)}_j:=}+
\sum\limits_{p=1}^{N}\sum\limits_{q=0}^{j-2k-3} b^{(m)}_{q,p}
\mathcal{L}_{j-q-2k}\psi_0^{(p)},
\end{align*}
and are represented as finite sums
\begin{equation*}
\widetilde{\psi}_j^{(m)}(\xi)=\sum\limits_\vs
\psi_{j,\vs,1}^{(m)}(\xi')\psi_{j,\vs,2}^{(m)}(\xi_d),
\end{equation*}
where $\psi_{j,\vs,1}^{(m)}\in\mathfrak{V}$,
$\psi_{j,\vs,2}^{(m)}\in C_0^\infty[0,1]$,
$\psi_{j,\vs,2}^{(m)}(0)=\psi_{j,\vs,2}^{(m)}(1)=0$, and the
functions $\psi_{j,\vs,2}^{(m)}$ are orthogonal to $\sin\pi
n\xi_d$ in $L_2(0,1)$. The functions
$\widetilde{\Psi}_j^{(m)}\in \mathfrak{V}$ are solutions to the
equations
\begin{align*}
(G_n-\L)\widetilde{\Psi}_j^{(m)}=&\widetilde{f}^{(m)}_{j+2k}+
\sum\limits_{q=1}^{j-1} c^{(m)}_{j+2k-q}
\widetilde{\Psi}^{(m)}_q +\sum\limits_{q=1}^{j-3}
\sum\limits_{p=1}^{N} c^{(m)}_{j+2k-q} b^{(m)}_{q,p}
\Psi_0^{(p)}
\\
+&\sum\limits_{q=1}^{j-2}\sum\limits_{p=1}^{N} c^{(m)}_{j+2k-q}
b^{(m)}_{q-1,p} \Phi_1^{(p)}- \sum\limits_{p=1}^{N}
(\widetilde{f}^{(m)}_{j+2k},
\Psi_0^{(p)})_{L_2(\mathbb{R}^{d-1})} \Psi_0^{(p)},
\end{align*}
and are orthogonal to all $\Psi_0^{(l)}$, $l=1,\ldots,N$, in
$L_2(\mathbb{R}^{d-1})$. The constants $b^{(m)}_{j,p}$ and
$c^{(m)}_j$ are determined by the formulae
\begin{align*}
&
\begin{aligned}
&b^{(m)}_{l,-1}=0,\quad b^{(m)}_{0,l}=\d_{ml}, \quad
b_{j,m}^{(m)}=0,\quad j\geqslant 1,
\\
&b^{(m)}_{j,l}=\frac{
2(\widetilde{F}_{j+2k+2}^{(m)},\psi_0^{(l)})
+\sum\limits_{q=1}^{j-1} c^{(m)}_{j+2k-q+2} b_{q,l}^{(m)}
}{\tau_{2k+1}^{(m)}-\tau_{2k+1}^{(l)}},
\quad m\not=l,\quad j\geqslant 1,
\end{aligned}
\\
&c^{(m)}_{2k}=\L,\quad c^{(m)}_{2k+1}=-\tau^{(2k+1)},\quad
c^{(m)}_{2k+2}=-\tau^{(2k+2)}_m,
\\
&c^{(m)}_j=-2(\widetilde{F}_j^{(m)},\psi_0^{(m)})_{L_2(\Pi)},
\quad j\geqslant 2k+3.
\end{align*}
\end{lemma}

These lemmata can be proven by induction.
\begin{rmk}\label{rm3.1}
We observe that if $\L$ is simple, then $N=1$ and the hypothesis
of Lemma~\ref{lm3.7} is obviously true.
\end{rmk}

We denote
\begin{align*}
&\psi_{\e,s}^{(m)}(x):=\chi(x')\sqrt{H(x')}\sum\limits_{j=0}^{s}
\eta^j\psi_j^{(m)}\left(\frac{x'-\xm}{\eta},\frac{x_d+\e h_-(x')}{\e
H(x')}\right),
\\
&\l_{\e,s}^{(m)}:=\e^{-2} c_0^{(m)}+\e^{-2} \sum\limits_{j=2k}^s
\eta^j c_j^{(m)},\quad s\geqslant 2k.
\end{align*}

The next lemma follows from the construction of the functions
$\psi_j^{(m)}$ and the constants $c_j^{(m)}$.

\begin{lemma}\label{lm3.9}
The functions $\psi_{\e,s}^{(m)}$ solve the boundary value
problems
\begin{equation}\label{3.30}
-(\D^{D}_{\Om_\e}+\l_{\e,s}^{(m)})\psi_{\e,s}^{(m)}=g_{\e,s}^{(m)},
\quad m=1,\ldots,N,
\end{equation}
where the right-hand sides satisfy the estimate
\begin{equation}\label{3.31}
\|g_{\e,s}^{(m)}\|_{L_2(\Om_\e)}=\Odr(\eta^{s-\frac{3k-d}{2}
-2}),\quad m=1,\ldots,N.
\end{equation}
\end{lemma}

We rewrite problem (\ref{3.30}) as
\begin{equation*}
\psi_{\e,s}^{(m)}=A_\e
\psi_{\e,s}^{(m)}+\frac{1}{1+\l_{\e,s}^{(m)}} A_\e
g_{\e,s}^{(m)},
\end{equation*}
where $A_\e:=(-\D^D_{\Om_\e}+1)^{-1}$. This operator is
self-adjoint, compact and satisfies the estimate
$\|A_\e\|\leqslant 1$. In view of this estimate and (\ref{3.31})
we have
\begin{equation*}
\left\|\frac{1}{1+\l_{\e,s}^{(m)}} A_\e g_{\e,s}^{(m)}
\right\|_{L_2(\Om_\e)}\leqslant C_{m,s}
\eta^{s-\frac{3k-d}{2}+2k}, \quad m=1,\ldots,N,
\end{equation*}
where $C_{m,s}$ are constants. We apply Lemma~1.1  to conclude
that there exists an eigenvalue $\vr_s^{(m)}(\e)$ of $A_\e$ such
that
\begin{equation*}
|\vr_s^{(m)}(\e)-(1+\l_{\e,s}^{(m)})^{-1}|\leqslant C_{m,s}
\eta^{s-\frac{3k-d}{2}+2k}, \quad m=1,\ldots,N.
\end{equation*}
Hence, the number
$\l^{(m)}_s(\e):=\big(\vr_s^{(m)}(\e)\big)^{-1}-1$ is an
eigenvalue of the operator $-\D^D_{\Om_\e}$, which satisfies the
identity
\begin{equation}\label{3.31a}
|\l^{(m)}_s(\e)-\l_{\e,s}^{(m)}|\leqslant \widetilde{C}_{m,s}
\eta^{s-\frac{7k-d}{2}-4}, \quad m=1,\ldots,N,
\end{equation}
where $\widetilde{C}_{m,s}$ are constants.

Let $\e^{(m)}_s$ be a monotone sequence such that
$\widetilde{C}_{m,s}\eta\leqslant \widetilde{C}_{m,s-1}$ as
$\e\leqslant \e^{(m)}_s$. We choose the eigenvalue
$\l_\e^{(m)}:=\l_{\e,s}^{(m)}$ as
$\e\in[\e_s^{(m)},\e_{s+1}^{(m)})$. The inequality (\ref{3.31a}) 
implies that the eigenvalue $\l_\e^{(m)}$ has the asymptotics
(\ref{2.2}). We employ Lemma~1.1 in \cite[Ch. I\!I\!I, Sec.
1.1]{IOS} once again with $\a=C_{m,s}\eta^{s-\frac{3k-d}{2}}$,
$d=\sqrt{\a}$. It yields that there exists a linear combination
$\psi_s^{(m)}(x,\e)$ of the eigenfunctions of $-\D^D_{\Om_\e}$
associated with the eigenvalues lying in
$[\l_\e^{(m)}-d,\l_\e^{(m)}+d]$ such that
\begin{equation*}
\|\psi_s^{(m)}(\cdot,\e)-\psi_{\e,s}^{(m)}\|_{L_2(\Om_\e)}=
\Odr(\eta^{\frac{2s-3k+d}{4}}),\ldots m=1,\ldots,N.
\end{equation*}
Since the functions $\psi_{\e,s}^{(m)}$ are linearly independent
for different $m$, the same is true for
$\psi_s^{(m)}(\cdot,\e)$, if $s$ is large enough. Thus, the
total multiplicity of the eigenvalues $\l_\e^{(m)}$ is at least
N. The proof is complete.

\section{Proof of Theorem~\ref{th2.2}\label{proofthm2}}

In order to prove Theorem~\ref{th2.2} we need to ensure that, for
sufficiently small $\e$, the asymptotic expansions for $\lambda_{1,m}$,
$m=1,\dots,N$ provided by Theorem~\ref{th2.1} do correspond to the first
$N$ eigenvalues of $-\Delta_{\Omega_{\e}}^{D}$ (counting multiplicities).
In~\cite{BF} this was done by means of adapting the proof of Theorem 1.1
in~\cite{frso} from the situation where $h_{-}=0$ to our case. In the
present context we need to show that, under the conditions for $h_{\pm}$,
this result may be extended to $d$ dimensions. There are two important
points that should be stressed here. On the one hand, we are assuming
$C^\infty$ regularity in a neighbourhood of the point of global maximum,
and thus do not have to deal with what could now be more complex
regularity issues at this point. On the other
hand, since the proof of eigenvalue convergence given in~\cite{frso}
is based on convergence in the norm, it is not affected by details related
to the possible higher multiplicities as was the case in the derivation of
the formulae in the previous section.

While still using the notation defined in Section~\ref{stat}, we also
refer to the notation in~\cite{frso}. In particular, the function $h$
and the operator ${\bf H}$ defined defined there correspond to our width
function $H$ and operator $G_{n}$, respectively. We begin by assuming $H$
to be strictly positive in $\overline{\omega}$. Let thus
\[
\psi(x',x_{d}) = \psi_{\chi}(x',x_{d}) = \chi(x')\sqrt{\frac{\ds 2}{\e H(x')}}
\sin\left[\frac{\ds \pi(x_{d}+\e h_{-}(x'))}{\ds \e H(x')}\right]
\]
As in~\cite{frso}, we have
\[
\|\psi_{\chi}(x',x_{d})\|_{L_{2}(\Omega_{\e})} = \dint_{\omega}\chi^2(x')dx',
\]
while now
\[
\dint_{\omega}\dint_{-\e h_{-}(x')}^{\e h_{+}(x')}\left|\nabla\psi_{\chi}(x',x_{d})\right|^{2}
=\dint_{\omega}\left|\nabla\chi(x')\right|^{2}+\left(\frac{\ds \pi^2}{\ds \e^2H^2(x')}+v(x')
\right)\chi^{2}(x')dx',
\]
with
\[
v(x') = \frac{\ds \pi^2}{\ds H^{2}(x')}\left[ \left|\frac{\ds 1}{\ds 2}\nabla H(x')-
\nabla h_{-}(x')\right|^{2}+\frac{\ds 1}{\ds 4}\left(\frac{\ds 1}{\ds 3}+
\frac{\ds 1}{\ds \pi^2}\right)\left|\nabla H(x')\right|^2\right].
\]
In the notation of~\cite{frso}, the potential $W_{\e}$ appearing in the quadratic form
$q_{\e}[\chi]$ (equation~(1.4) on page $3$ in that paper) is now defined by
\[
W_{\e}(x') = \frac{\ds \pi^2}{\ds \e^2}\left[\frac{\ds 1}{\ds H^{2}(x')}-
\frac{\ds 1}{\ds H^{2}(\xm)}\right]+v(x').
\]
We consider the scaling $x'=e^{\alpha}t$ as before, which causes the
domain $\omega$ to be scaled to $\omega_{\e}=e^{\alpha}\omega$. Then
the proofs of Lemma 2.1 and Theorem 1.2 go through with minor
changes (note that $m=2k$, while $I$ and $I_{\e}$ should be changed
by $\omega$ and $\omega_{\e}$, respectively). Similar remarks apply
to the proofs in Section 4 of~\cite{frso} leading to the proof of
Theorem 1.3, except that due to regularity we do not need to worry
about separating the domain into different parts as was necessary
there for the intervals $I_{\e}$. 

Finally, we relax the condition on the strict positivity of $H$ mentioned above. This
again follows in a similar fashion to what was done in Section 6.1 of~\cite{frso}.

We are now in conditions to proceed to the proof of (\ref{2.5}). In the case considered
the lowest eigenvalue of $G_1$ is $\L=\sum_{j=1}^d\tht_j$, while
the associated eigenfunction is given by (\ref{2.8}). This
proves the formula for $c_2^{(1,1)}$. In view of
Remark~\ref{rm3.1}, we can employ Lemma~\ref{lm3.7} to calculate
$c_3^{(1,1)}$, $c_4^{(1,1)}$. Since $\Psi_0$ is even w.r.t. each
$\xi_i$, $i=1,\ldots,d-1$, and $H_3(-\xi')=-H_3(\xi')$, we
conclude by (\ref{3.21a}) that $T^{(3)}_{11}=0$. By
Theorem~\ref{th2.1} it yields that $c^{(1,1)}_3=0$.

The equation (\ref{3.11a}) for $\widetilde{\Psi}_1$ with $j=3$ reads
as follows
\begin{equation}\label{4.2}
(G_1-\L)\widetilde{\Psi}_1=\frac{2\pi^2}{H_0^3}H_3\Psi_0.
\end{equation}
We seek the solution as $\widetilde{\Psi}_1=R\Psi_0$, where $R$
is a polynomial of the form
\begin{equation}\label{4.3}
R(\xi'):=\sum\limits_{p,q,j=1}^{d-1}C_{pqj}\xi_p\xi_q\xi_j+
\sum\limits_{j=1}^{d-1} C_j\xi_j,
\end{equation}
where $C_{pqj}$, $C_j$ are constants to be found, and $C_{pqj}$
are invariant under each permutation of the indices $p$, $q$,
$j$. We also note that such a choice of $R$ ensures that
$(\widetilde{\Psi}_1,\Psi_0)_{L_2(\mathbb{R}^{d-1})}=0$. We
substitute (\ref{4.3}) and the formula for $\widetilde{\Psi}_1$
into (\ref{4.2}) taking into account (\ref{2.7}),
\begin{multline*}
2\sum\limits_{p,q,j=1}^{d-1} (\tht_p+\tht_q+\tht_j) C_{pqj}
\xi_p\xi_q\xi_j+2\sum\limits_{j=1}^{d-1} \tht_j C_j\xi_j+
6\sum\limits_{p,j=1}^{d-1} C_{ppj}\xi_j
\\
=
-\frac{2\pi^2}{H_0^3}\sum\limits_{p,q,j=1}^{d-1}
\b_{pqj}\xi_p\xi_q\xi_j.
\end{multline*}
It yields the formulae
\begin{equation}\label{4.4}
C_{pqj}=-\frac{\pi^2\b_{pqj}}{H_0^3(\tht_p+\tht_q+\tht_j)},\quad
C_j=\frac{3}{2}\sum\limits_{p=1}^{d-1}
\frac{\pi^2\b_{ppj}}{H_0^3\tht_j(2\tht_p+\tht_j)}.
\end{equation}

It is easy to check that
\begin{equation}\label{4.5}
\begin{aligned}
&Q_1^{(i)}(\xi')=-\frac{1}{H_0}\frac{\p H_2}{\p x_i}(\xi'),\quad
P_0^{(0)}=\frac{1}{2H_0}\D_{x'} H_2,
\\
&P_4^{(d)}(\xi')=H_0^{-4} \big(3H_2^2(\xi')-2H_0H_4(\xi')\big).
\end{aligned}
\end{equation}
Employing these identities, we write the formula for
$c_4^{(1,1)}$ from Lemma~\ref{lm3.7}
\begin{align*}
c^{(1,1)}_4=&-2(\widetilde{F}_4,\psi_0)_{L_2(\Pi)}
=-2(\mathcal{L}_2\psi_0,\psi_0)_{L_2(\Pi)}
-2(\mathcal{L}_1\widetilde{\Psi}_1\sin\pi \xi_d,\psi_0)_{L_2(\Pi)}
\\
=&\pi^2 (P^{(d)}_4\Psi_0,\Psi_0)_{L_2(\mathbb{R}^{d-1})}-
(P_0^{(0)}\Psi_0,\Psi_0)_{L_2(\mathbb{R}^{d-1})}
\\
&+4\pi \sum\limits_{i=1}^{d-1}
\left(Q_1^{(i)}\frac{\p\Psi_0}{\p\xi_i}\sin\pi
\xi_d,\Psi_0\xi_d\cos\pi\xi_d\right)_{L_2(\Pi)}
\\
&-2\sum\limits_{i=1}^{d-1}
\big(K_0^{(i)}\big)^2\left(\frac{\p^2\psi_0}{\p\xi_d^2},\psi_0\right)_{L_2(\Pi)}
+\pi^2 (P^{(d)}_3\widetilde{\Psi}_1,\Psi_0)_{L_2(\mathbb{R}^{d-1})}
\\
=&\pi^2 (P^{(d)}_4\Psi_0,\Psi_0)_{L_2(\mathbb{R}^{d-1})}
+\frac{1}{2H_0}\sum\limits_{i=1}^{d-1} \a_i^2 -
\sum\limits_{i=1}^{d-1} \left(Q_1^{(i)}\frac{\p\Psi_0}{\p\xi_i},
\Psi_0\right)_{L_2(\mathbb{R}^{d-1})}
\\
&+\frac{\pi^2}{H_0^2} \sum\limits_{i=1}^{d-1} \left(\frac{\p h_1}{\p
x_i}(\xm)\right)^2+\pi^2
(P^{(d)}_3\widetilde{\Psi}_1,\Psi_0)_{L_2(\mathbb{R}^{d-1})}
\\
=&\pi^2 ((P^{(d)}_4+P^{(d)}_3)\Psi_0,\Psi_0)_{L_2(\mathbb{R}^{d-1})}
+\frac{\pi^2}{H_0^2} \sum\limits_{i=1}^{d-1} \left(\frac{\p h_1}{\p
x_i}(\xm)\right)^2.
\end{align*}
We substitute (\ref{4.3}), (\ref{4.4}), (\ref{4.5}) into this
identity and arrive at the desired formula for $c_4^{(1,1)}$.

\section{The $d-$dimensional ellipsoid}

As an application of our results, we will derive the expression~(\ref{lellipsoid})
for the asymptotic expansion for the first eigenvalue for a general ellipsoid. From
the equation defining the
boundary of $\mathcal{E}$ and assuming that, as mentioned in the Introduction, we
are doing the scaling along the $x_{d}$ axis, we have
\[
h_{\pm}(x') = a_{d}\left[1-\left(\frac{\ds x_{1}}{\ds a_{1}}\right)^2-\dots
-\left(\frac{\ds x_{d-1}}{\ds a_{d-1}}\right)^2\right]^{1/2},
\]
while $H(x')=2h_{\pm}(x')$. We thus have $\xm$ located at the origin and $H_{0}=2a_{d}$.
Expanding $H$ around $\xm$ we have
\[
\begin{array}{lll}
H(x') & = & 2a_{d}-a_{d}\left[\left(\frac{\ds x_{1}}{\ds a_{1}}\right)^2+\dots
+\left(\frac{\ds x_{d-1}}{\ds a_{d-1}}\right)^{2}\right]\eqskip
& & \hspace*{5mm}-\frac{\ds a_{d}}{\ds 4}\left[ \left(\frac{\ds x_{1}}{\ds a_{1}}\right)^4+\dots
+\left(\frac{\ds x_{d-1}}{\ds a_{d-1}}\right)^{4}\right.\eqskip
& & \hspace*{10mm} + \left.
2\left(\frac{\ds x_{1}x_{2}}{\ds a_{1}a_{2}}\right)^2+
2\left(\frac{\ds x_{1}x_{3}}{\ds a_{1}a_{3}}\right)^2\dots
+2\left(\frac{\ds x_{d-2}x_{d-1}}{\ds a_{d-2}a_{d-1}}\right)^2\right]\eqskip
& & \hspace*{15mm} + \dots,
\end{array}
\]
yielding $H_{k}=h_{k}=0$ for odd $k$ and
\[
\begin{array}{lll}
H_{2}(x') & = & -a_{d}\dsum_{i=1}^{d-1}\left(\frac{\ds x_{i}}{\ds a_{i}}\right)^{2}\eqskip
H_{4}(x') & = & -\frac{\ds a_{d}}{\ds 4}\dsum_{i=1}^{d-1}\dsum_{j=1}^{d-1}
\left(\frac{\ds x_{i}x_{j}}{\ds a_{i}a_{j}}\right)^2.
\end{array}
\]
Hence
\[
\begin{array}{llll}
\alpha_{i}=\frac{\ds \sqrt{2a_{d}}}{\ds a_{i}}, &
\theta_{i}=\frac{\ds \pi}{\ds 2a_{i}a_{d}}
\end{array}
\]
and
\[
\psi_{0}(x') = \frac{\ds 2^{\frac{1-d}{4}}a_{d}^{\frac{1-d}{4}}}
{\ds (a_{1}\dots a_{d-1})^{1/2}}e^{-\frac{\ds \pi}{4a_{d}}\left(
\frac{x_{1}^2}{a_{1}}+\dots+\frac{x_{d-1}^2}{a_{d-1}}\right)}.
\]
Note that since $H_{3}$ is identically zero, there is no need to compute $\widetilde{\Psi}_1$.
It is now straightforward to obtain
\[
c_{0}^{(1,1)} = \frac{\ds \pi^2}{\ds 4a_{d}^2}\quad \mbox{ and
}\quad c_{2}^{(1,1)} = \frac{\ds \pi}{\ds
2a_{d}}\dsum_{i=1}^{d-1}\frac{\ds 1}{\ds a_{i}}.
\]
It remains to compute
\[
\begin{array}{lll}
c_{4}^{(1,1)} & = & \frac{\ds \pi^2}{\ds 16a^{2}_{d}}\left(\left[ 3
\left(\dsum_{i=1}^{d-1}\left(\frac{\ds x_{i}}{\ds a_{i}}\right)^2\right)^2
\right.\right.\eqskip
& & \hspace*{5mm}\left.\left.+
\dsum_{i=1}^{d-1}\dsum_{j=1}^{d-1}\left(\frac{\ds x_{i}x_{j}}{\ds a_{i}a_{j}}\right)^2
\right]\psi_{0}(x'),\psi_{0}(x')\right)_{L_2(\mathbb{R}^{d-1})}\eqskip
& = & \frac{\ds \pi^{2}}{\ds 2^{\frac{d+3}{2}}a_{d}^{\frac{d+3}{2}}(a_{1}\dots a_{d-1})^{1/2}}
\eqskip
& & \hspace*{10mm}\times
\dsum_{i=1}^{d-1}\dsum_{j=1}^{d-1}\dint_{\R^{d-1}}\left(\frac{\ds x_{i}x_{j}}{\ds a_{i}a_{j}}
\right)^{2}e^{-\frac{\pi}{2a_{d}}\left(\frac{x_{1}^2}{a_{1}}+\dots+
\frac{x_{d-1}^2}{a_{d-1}}\right)}dx'
\end{array}
\]
which, after some further simplifications, yields the desired result.


\begin{thebibliography}{999999}
\bibitem[BF]{BF} D. Borisov and P. Freitas, Singular asymptotic
expansions for Dirichlet eigenvalues and eigenfunctions of the
Laplacian on thin planar domains, Ann. Inst. H. Poincar\'{e} Anal.
Non--Lin\'{e}aire {\bf 26} (2009) 547-560.

\bibitem[D]{D} E.~B. Davies,
\newblock {\em Heat kernels and spectral theory}.
\newblock Cambridge University Press, Cambridge, 1989.

\bibitem[F]{frei}
P. Freitas, Precise bounds and asymptotics for the first
Dirichlet eigenvalue of triangles and rhombi, J. Funct. Anal.
{\bf 251} (2007), 376-398. doi:101016/jfa.2007.04.012.

\bibitem[FS]{frso}
L. Friedlander and M. Solomyak, On the spectrum of the Dirichlet
Laplacian in a narrow strip, to appear in Israel J. Math..

\bibitem[G]{G} I. M. Glazman,
{\em Direct Methods of Qualitative Spectral Analysis of Singular
Differential Operators}, London, Oldbourne Press, 1965.

\bibitem[H]{henr} D. Henry,
{\em Perturbation of the Boundary in Boundary-Value Problems
of Partial Differential Equations}, With editorial assistance
from J. Hale and A.L. Pereira. London Mathematical Society Lecture
Note Series, {\bf 318}, Cambridge University Press, Cambridge, 2005


\bibitem[J]{jose}
D. D. Joseph, Parameter and domain dependence of eigenvalues of
elliptic partial differential equations, Arch. Rational Mech.
Anal.  {\bf 24} (1967), 325--351.

\bibitem[OSY]{IOS} O.~A. Olejnik, A. S. Shamaev and G. A. Yosifyan,
\emph{Mathematical problems in elasticity and homogenization.}
Studies in Mathematics and its Applications. 26. Amsterdam etc.:
North-Holland, 1992.

\end{thebibliography}
\end{document}